\documentclass[11pt]{article}
\usepackage{amssymb}
\newtheorem{precor}{{\bf Corollary}}

\newtheorem{precon}{{\bf Conjecture}}

\newenvironment{con}{\begin{precon}{\hspace{-0.5
               em}{\bf.\ }}}{\end{precon}}
\newtheorem{predefin}{{\bf Definition}}

\newenvironment{defin}[1]{\begin{predefin}{\hspace{-0.5
                   em}{\bf.\ }}{\rm #1}\hfill{$\spadesuit$}}{\end{predefin}}
\newtheorem{preexm}{{\bf Example}}

\newtheorem{preappl}{{\bf Application}}

\newtheorem{prelem}{{\bf Lemma}}

\newtheorem{preproof}{{\bf Proof.\ }}

\newenvironment{proof}[1]{\begin{preproof}{\rm
               #1}\hfill{$\blacksquare$}}{\end{preproof}}
\newtheorem{prethm}{{\bf Theorem}}

\newenvironment{thm}{\begin{prethm}{\hspace{-0.5
               em}{\bf.\ }}}{\end{prethm}}
\newtheorem{prealphthm}{{\bf Theorem}}

\newenvironment{alphthm}{\begin{prealphthm}{\hspace{-0.5
               em}{\bf.\ }}}{\end{prealphthm}}
\newtheorem{prealphlem}{{\bf Lemma}}

\newenvironment{alphlem}{\begin{prealphlem}{\hspace{-0.5
               em}{\bf.\ }}}{\end{prealphlem}}
\newtheorem{prepro}{{\bf Proposition}}

\newtheorem{prequ}{{\bf Question}}

\newenvironment{qu}{\begin{prequ}{\hspace{-0.5
               em}{\bf.\ }}}{\end{prequ}}
\newtheorem{preprb}{{\bf Problem}}

\newenvironment{prb}{\begin{preprb}{\hspace{-0.5
               em}{\bf.\ }}}{\end{preprb}}
\def\conct[#1,#2]{\mbox {${#1} \leftrightarrow {#2}$}}
\def\dconct[#1,#2]{\mbox {${#1} \rightarrow {#2}$}}
\def\deg[#1,#2]{\mbox {$d_{_{#1}}(#2)$}}
\def\mindeg[#1]{\mbox {$\delta_{_{#1}}$}}
\def\maxdeg[#1]{\mbox {$\Delta_{_{#1}}$}}
\def\outdeg[#1,#2]{\mbox {$d_{_{#1}}^{^+}(#2)$}}
\def\minoutdeg[#1]{\mbox {$\delta_{_{#1}}^{^+}$}}
\def\maxoutdeg[#1]{\mbox {$\Delta_{_{#1}}^{^+}$}}
\def\indeg[#1,#2]{\mbox {$d_{_{#1}}^{^-}(#2)$}}
\def\minindeg[#1]{\mbox {$\delta_{_{#1}}^{^-}$}}
\def\maxindeg[#1]{\mbox {$\Delta_{_{#1}}^{^-}$}}
\def\isdef{\mbox {$\ \stackrel{\rm def}{=} \ $}}
\def\dre[#1,#2,#3]{\mbox {${\cal E}_{_{#3}}(#1,#2)$}}
\def\var[#1,#2]{\mbox {${\rm Var}_{_{#1}}(#2)$}}
\def\ls[#1]{\mbox {$\xi^{^{#1}}$}}
\def\hom[#1,#2]{\mbox {${\rm Hom}({#1},{#2})$}}
\def\onvhom[#1,#2]{\mbox {${\rm Hom^{v}}(#1,#2)$}}
\def\onehom[#1,#2]{\mbox {${\rm Hom^{e}}(#1,#2)$}}
\def\core[#1]{\mbox {$#1^{^{\bullet}}$}}
\def\cay[#1,#2]{\mbox {${\rm Cay}({#1},{#2})$}}
\def\cays[#1,#2]{\mbox {${\rm Cay_{s}}({#1},{#2})$}}
\def\dirc[#1]{\mbox {$\stackrel{\rightarrow}{C}_{_{#1}}$}}
\def\cycl[#1]{\mbox {${\bf Z}_{_{#1}}$}}

\date{}
\begin{document}
\footnotetext[1]{This research was partially supported by Shahid
Beheshti University.}
\footnotetext[2]{Correspondence should be addressed to {\tt
hhaji@sbu.ac.ir}.}
\begin{center}
{\Large \bf  On Colorings of Graph Powers}\\
\vspace*{0.5cm}
{\bf Hossein Hajiabolhassan}\\
{\it Department of Mathematical Sciences}\\
{\it Shahid Beheshti University}\\
{\it P.O. Box {\rm 19834}, Tehran, Iran}\\
{\tt hhaji@sbu.ac.ir}\\
\end{center}
\begin{abstract}
\noindent In this paper, some results concerning the colorings of
graph powers are presented. The notion of helical graphs is
introduced. We show that such graphs are hom-universal with
respect to high odd-girth graphs whose $(2t+1)$st power is
bounded by a Kneser graph. Also, we consider the problem of
existence of homomorphism to odd cycles. We prove that such
homomorphism to a $(2k+1)$-cycle exists if and only if the
chromatic number of the $(2k+1)$st power of $S_2(G)$ is less than
or equal to 3, where $S_2(G)$ is the 2-subdivision of $G$. We
also consider Ne\v set\v ril's Pentagon problem. This problem is
about the existence of high girth cubic graphs which are not
homomorphic to the cycle of size five. Several problems which are
closely related to Ne\v set\v
ril's problem are introduced and their relations are presented.\\

\end{abstract}
\section{Introduction}
Throughout this paper we only consider finite graphs. A {\it
homomorphism} $f: G \longrightarrow H$ from a graph $G$ to a
graph $H$ is a map $f: V(G) \longrightarrow V(H)$ such that $uv
\in E(G)$ implies $f(u)f(v) \in E(H)$. The existence of a
homomorphism is indicated by the symbol $G \longrightarrow H$.
Two graphs $G$ and $H$ are homomorphically equivalent if  $G
\longrightarrow H$ and $H \longrightarrow G$. Also, the symbol
$\hom[G,H]$ is used to denote the set of all homomorphisms from
$G$ to $H$ (for more on graph homomorphisms
see \cite{DAHA1, DAHA2, HT, HN}).\\
 If $n$ and $d$ are positive integers with  $n \geq
2d$, then the {\em circular complete graph} $K_{_{(n,d)}}$ is the
graph with vertex set $\{v_{_{0}}, v_{_{1}}, \ldots, v_{_{n-1}}\}$
in which $v_{_{i}}$ is connected to $v_{_{j}}$ if and only if $d
\leq |i-j| \leq n-d$. A graph $G$ is said to be $(n,
d)$-colorable if $G$ admits a homomorphism to $K_{_{(n,d)}}$. The
{\em circular chromatic number} (also known as the {\it star
chromatic number} \cite{VINCE}) $\chi_{_{c}}(G)$ of a graph $G$
is the minimum of those ratios $\frac{n}{d}$ for which
$gcd(n,d)=1$ and such that $G$ admits a homomorphism to
$K_{_{(n,d)}}$. It can be shown that one may only consider
onto-vertex homomorphisms \cite{ZH}. We denote by $[m]$ the set
$\{1, 2, \ldots, m\}$, and denote by ${[m] \choose n}$ the
collection of all $n$-subsets of $[m]$. For a given subset
$A\subseteq [m]$, the complement of $A$ in $[m]$ is denoted by
$\overline{A}$. The {\em Kneser graph} $KG(m,n)$ is the graph with
vertex set ${[m] \choose n}$, in which $A$ is connected to $B$ if
and only if $A \cap B = \emptyset$. It was conjectured by Kneser
\cite{kne} in 1955, and proved by Lov\'{a}sz \cite{lov} in 1978,
that $\chi(KG(m,n))=m-2n+2$. A subset $S$ of $[m]$ is called
$2$-{\em stable} if  $2 \le |x-y| \le m-2$ for all distinct
elements $x$ and $y$ of $S$.  The {\em Schrijver graph} $SG(m,n)$
is the subgraph of $KG(m,n)$ induced by all $2$-stable
$n$-subsets of $[m]$. It was proved by Schrijver \cite{sch} that
$\chi(SG(m,n))=\chi(KG(m,n))$ and that every proper subgraph of
$SG(m,n)$ has a chromatic number smaller than that of $SG(m,n)$.
The {\it fractional chromatic number}, $\chi_{_{f}}(G)$, of a
graph $G$ is defined as
$$\chi_{_{f}}(G) \isdef \inf \{\frac{m}{n} \ \ | \ \ \hom[G,KG(m,n)]
\not = \emptyset \}.$$

For more about the fractional coloring see \cite{SCUL}. The {\it
local chromatic number} of a graph was defined in \cite{ERFU} as
the minimum number of colors that must appear within distance $1$
of a vertex. Here is the formal definition.

\begin{defin}{
The local chromatic number  $\psi(G)$ of a graph $G$ is
$$\psi(G)\isdef {\displaystyle \min_c}{\displaystyle \max_{v\in V(G)}}|\{c(v): u\in N(v)\}|+1,$$
where the minimum is taken over all proper colorings $c$ of $G$
and $N(v)=N_G(v)$ denotes the neighborhood of a vertex $v$ in a
graph $G$.}
\end{defin}

It is easy to verify that for any graph $G$, $\psi(G) \leq
\chi(G)$. Also, it was shown in \cite{KOPISI} that $\chi_f(G)\leq
\psi(G)$ holds for any graph $G$.

For a graph $G$, let $G^{^{(k)}}$ be the $k$th power of $G$, which
is obtained on the vertex set $V(G)$, by connecting any two
vertices $u$ and $v$ for which there exists a walk of length $k$
between $u$ and $v$ in $G$. Note that the $k$th power of a simple
graph is not necessarily a simple graph itself. For instance, the
$k$th power may have loops on its vertices provided that $k$ is
an even integer. For a given graph $G$ with $og(G)\geq 7$, the
chromatic number of $G^{(5)}$ provides an upper bound for the
local chromatic number of $G$. In \cite{SITA}, it was proved
$\psi(G) \leq \lfloor {m\over 2}\rfloor +2$ whenever
$\chi(G^{(5)})\leq m$.

The following simple and useful lemma was proved and used
independently in \cite{DAHA4, NEME1, TA}.

\begin{alphlem}
\label{M2} Let $G$ and $H$ be two simple graphs such that
$\hom[G,H] \not = \emptyset$. Then for any positive integer $k$,
$\hom[G^{^{(k)}},H^{^{(k)}}] \not = \emptyset$.
\end{alphlem}

Note that Lemma \ref{M2} trivially holds whenever $H^{^{(k)}}$
contains a loop, e.g., when $k=2$. As immediate consequences of
Lemma~\ref{M2}, we obtain $\chi_c(P)=\chi(P)$ and
$\hom[C,C_{_{7}}]=\emptyset$, where $P$ and $C$ are the Petersen
and the Coxeter graphs, respectively, see  \cite{DAHA4}.

In what follows we are concerned with some results concerning the
colorings of graph powers. First, The notion of helical graphs is
introduced. We show that such graphs are hom-universal with
respect to high odd-girth graphs whose $(2t+1)$st power is
bounded by a Kneser graph. Then, we consider the problem of
existence of homomorphism to odd cycles. We prove that such
homomorphism to a $(2k+1)$-cycle exists if and only if the
chromatic number of the $(2k+1)$st power of $S_2(G)$ is less than
or equal to 3, where $S_2(G)$ is the 2-subdivision of $G$. We
also consider Ne\v set\v ril's Pentagon problem. This problem is
about the existence of high girth cubic graphs which are not
homomorphic to the cycle of size five. Several problems which are
closely related to Ne\v set\v ril's problem are introduced and
their relations are presented.
\section{Helical Graphs}
For a given class ${\cal C}$ of graphs, a graph $U$ is called
hom-universal with respect to ${\cal C}$ if for any $G\in {\cal
C}$, $\hom[G,U] \not = \emptyset$, in which case the class ${\cal
C}$ is said to be bounded by the graph $U$. The problem of the
existence of a bound with some special properties, for a given
class of graphs, has been a subject of study in graph
homomorphism. In the following definition, we introduce a new
family of hom-universal graphs, namely the family $H(m,n,k)$ of
the helical graphs.

\begin{defin}{Let $m, n,$ and $k$ be positive integers with $m\geq 2n$.
Set $H(m,n,k)$  to be the {\it helical graph} whose vertex set
contains all $k$-tuples $(A_1,\ldots ,A_k)$ such that for any
$1\leq r\leq k$, $A_r\subseteq [m], |A_1|=n, |A_r|\geq n$ and for
any $s\leq k-1$ and $t \leq k-2$, $A_s\cap A_{s+1}=\emptyset,
A_t\subseteq A_{t+2}$.
Also, two vertices $(A_1,\ldots ,A_k)$ and $(B_1,\ldots ,B_k)$ of
$H(m,n,k)$ are adjacent if for any $1\leq i, j+1\leq k$, $A_i
\cap B_i=\emptyset, A_j \subseteq B_{j+1}$, and $B_j \subseteq
A_{j+1}$.
}
\end{defin}

 Note that $H(m,1,1)$ is the complete graph $K_m$ and $H(m,n,1)$
is the Kneser graph $KG(m,n)$. It is easy to verify that if $m >
2n$, then the odd-girth of $H(m,n,k)$ is greater than or equal to
$2k+1$.

For a given graph $G$ and $v\in V(G)$, set
$$N_i(v)\isdef \{u|{\rm there\ is\ a\ walk\ of\ length}\ i\ {\rm joining}\
u\ {\rm and}\ v\}.$$ Also, for a coloring $c: V(G)\longrightarrow
{[m]\choose n}$, define
$$c(N_i(v))\isdef {\displaystyle \bigcup_{u \in N_i(v)}c(u)}.$$

The chromatic number of graph powers has been studied in the
literature (see \cite{BAST, DAHA4, GYJE, RN, SITA, TA}). In the
theorem below, we show that the helical graphs are hom-universal
graphs with respect to the family of high odd-girth graphs whose
$(2k-1)$st power is bounded by a Kneser graph.

\begin{thm}\label{homb} Let $G$ be a non-empty graph with
odd-girth at least $2k+1$. Then we have
$\hom[G^{(2k-1)},KG(m,n)]\not =\emptyset$ if and only if
$\hom[G,H(m,n,k)] \not = \emptyset.$
\end{thm}
\begin{proof}{First, let $c\in \hom[G^{(2k-1)},KG(m,n)]$. If $v$ is an isolated
vertex of $G$, then consider an arbitrary vertex, say $f(v)$, of
$H(m,n,k)$. For any non-isolated vertex $v\in V(G)$, define
$$f(v)\isdef (c(v), c(N_1(v)), c(N_2(v)), \ldots ,c(N_{k-1}(v))).$$
If $i\leq j$ and $i\equiv j\ mod\ 2$, we have $N_i(v)\subseteq
N_j(v)$, implying that $c(N_i(v))\subseteq c(N_j(v))$. Also, since
$c$ is a homomorphism from  $G^{(2k-1)}$ to $KG(m,n)$, for any
$i\leq j \leq k-1$ and $i\not \equiv j\ mod\ 2$, we obtain
$c(N_i(v))\cap c(N_j(v))=\emptyset$. Hence, for any vertex $v\in
V(G)$, $f(v) \in V(H(m,n,k))$. Moreover, for any $0\leq i, j+1
\leq k-1$, we have $N_i(v) \cap N_i(u)=\emptyset, N_j(v) \subseteq
N_{j+1}(u)$, and $N_j(u) \subseteq N_{j+1}(v)$ provided that  $u$
is adjacent to $v$. Hence, $f$ is a graph homomorphism from $G$
to $H(m,n,k)$.

Next, let $\hom[G,H(m,n,k)] \not = \emptyset$ and $f: G
\longrightarrow H(m,n,k)$. Assume $v\in V(G)$ and
$f(v)=(A_1,A_2,\ldots ,A_k)$. Define, $c(v)\isdef A_1$. Assume
further that $u,v \in V(G)$ such that there is a walk of length
$2t+1$ ($t\leq k-1$) between $u$ and $v$ in $G$, i.e., $uv\in
E(G^{(2k-1)})$. Consider adjacent vertices $u'$ and $v'$ such that
$u'\in N_t(u)$ and $v'\in N_t(v)$. Also, let
$f(v)=(A_1,A_2,\ldots ,A_k)$, $f(u)=(B_1,B_2,\ldots ,B_k)$,
$f(v')=(A'_1,A'_2,\ldots ,A'_k)$, and $f(u')=(B'_1,B'_2,\ldots
,B'_k)$. In view of the definition of the helical graph, we
obtain $A_1 \subseteq A'_{t+1}$ and $B_1 \subseteq B'_{t+1}$. On
the other hand, $A'_{t+1}\cap B'_{t+1}=\emptyset$, which yields
$c(v)\cap c(u) =\emptyset$. Thus, $\hom[G^{(2k-1)},KG(m,n)]\not
=\emptyset$, as desired.
}
\end{proof}

It was conjectured in {\rm \cite{MANANE}} that a class $\cal{C}$
of graphs is bounded by a graph $H$ whose odd-girth is at least
$2k + 1$ provided that the set $\{\chi(G^{(2k-1)})|G \in
\cal{C}\}$ is bounded and that all graphs in $\cal{C}$ have
odd-girth at least $2k +1$. It is worth noting that Theorem
\ref{homb} shows the above conjecture is true. This conjecture
however was proved by Tardif recently (personal communication,
see \cite{MANANE}).

It was proved by Schrijver \cite{sch} that $SG(m,n)$ is the
vertex-critical subgraph of $KG(m,n)$. Motivated by the
construction of the Schrijver graphs, we introduce a family of
subgraphs of the helical graphs.
\begin{defin}{Let $m, n,$ and $k$ be positive integers with $m\geq 2n$.
Define $SG(m,n,k)$ to be the induced subgraph of $H(m,n,k)$ whose
vertex set contains all $k$-tuples $(A_1,\ldots ,A_k)\in
V(H(m,n,k))$ such that for any $1\leq r\leq k$,
$A_r={\displaystyle \cup_s} B_s$, where every $B_s$ is a
$2$-stable $n$-subset of $[m]$.
}
\end{defin}

One can deduce the following theorem whose proof is almost
identical to that of Theorem \ref{homb} and the proof is omitted
for the sake of brevity.

\begin{thm}\label{shomb} Let $G$ be a non-empty graph with
odd-girth at least $2k+1$. Then, $\hom[G^{(2k-1)},SG(m,n)]\not
=\emptyset$ if and only if $\hom[G,SG(m,n,k)] \not = \emptyset.$
\end{thm}
 In \cite{GYJE}, it was proved that $\chi(H(m,1,2))=m$. Later  in \cite{BAST, SITA}, it was
shown $\chi(H(m,1,k))=m$. We would like to remark that the graph
$H(m,1,k)$ is defined in a completely different way in \cite{BAST,
SITA}. Simonyi and Tardos \cite{SITA} showed that
$\chi(H(m,1,k))=m$ by proving the existence of homomorphism from
$SG(a,b)$ to $H(m,1,k)$, where $a-2b+2=m$ and $a$ is sufficiently
large. Similarly, one can show that
$\chi(H(m,n,k))=\chi(SG(m,n,k))=m-2n+2$, where $m\geq 2n$.

\begin{alphlem}{\rm \cite{SITA}}\label{dist}
Let $u,v\subset [a]$ be two vertices of $SG(a,b)$. If there is a
walk of length $2s$ between $u$ and $v$ in $SG(a,b)$, then
$|u\setminus v | \leq s(a-2b+2)$.
\end{alphlem}

\begin{thm}\label{chrom}
Let $m, n,$ and $k$ be positive integers with $m\geq 2n$. The
chromatic number of the helical graph $H(m,n,k)$ is equal to
$m-2n+2$. Moreover, $\chi(SG(m,n,k))=m-2n+2$.
\end{thm}
\begin{proof}{For a given vertex vertex $v=(A_1,A_2,\ldots ,A_k) \in V(H(m,n,k))$, define $f(v)\isdef
A_1$. It is easy to check that $f$ is a graph homomorphism from
$H(m,n,k)$ to $KG(m,n)$. It follows that $\chi(SG(m,n,k))\leq
\chi(H(m,n,k))\leq m-2n+2$. Now, we prove that $m-2n+2$ is a lower
bound for the chromatic number of $SG(m,n,k)$. To this end, it
suffices to show, first, that for $a\isdef 2(k-1)m(m-2n+2)+m$ and
$b\isdef (k-1)m(m-2n+2)+n$, we have
$\hom[SG(a,b)^{(2k-1)},SG(m,n)]\not =\emptyset$. Then, Theorem
\ref{shomb} applies, and hence the assertion follows. Now, let
$[a]$ be partitioned into $m$ sets, each of which contains
$2(k-1)(m-2n+2)+1$ consecutive elements of $[a]$. In other words,
$[a]$ is partitioned into $m$ disjoint sets $D_1,\ldots,D_m$,
where each $D_i$ contains consecutive elements and $|D_i|
=2(k-1)(m-2n+2)+1$. Note that $b=(k-1)m(m-2n+2)+n$ and
${\displaystyle \sum_{i=1}^m}{(|D_i|-1)\over 2}=(k-1)m(m-2n+2)$.
Therefore, for every  $2$-stable subset $u$ of $[a]$ of size $b$,
there are at least $n$ indices $i_1,\ldots ,i_n$ such that $u$
contains $(k-1)(m-2n+2)+1$ elements of $D_{i_j}, 1\leq j \leq n$.
Note also that $D_i$ contains a unique subset of cardinality
$(k-1)(m-2n+2)+1$ which does not contain any two consecutive
elements. Use $E_i$ to denote this unique subset of $D_i$, which
is readily seen to consist of the smallest elements of $D_i$, the
third smallest elements of $D_i$, and so on and so forth. For any
vertex  $u\in SG(a, b)$, we define a coloring $c$ by choosing $n$
indices $i_j$ ($1\leq j \leq n$) such that $E_{i_j}\subseteq u$
and we set $c(u)\isdef \{i_1,\ldots ,i_n\}$. Since $u$ is a
$2$-stable subset of $[a]$, it is easy to verify that $c(u)$ is a
$2$-stable subset of $[m]$ too. One needs to show that for any two
vertices $u$ and $v$ for which there is a walk of length $2r-1$
between them, where $1\leq r\leq k$, we have $c(u)\cap
c(v)=\emptyset$. To prove this, suppose that $i\in c(v)$ and
$v=v_0,v_1, \ldots ,v_{2r-1}=u$ be a walk between $u$ and $v$,
where $1\leq r\leq k$. By Lemma \ref{dist}, $|v\setminus v_{2r-2}|
\leq (k-1)(m-2n+2)$. In particular, $v_{2r-2}$ contains all but at
most $(k-1)(m-2n+2)$ elements of $E_i$. As
$|E_i|=(k-1)(m-2n+2)+1$, we see that $v_{2r-2}\cap E_i\not =
\emptyset$. Thus, the set $u$, which is disjoint from $v_{2r-2}$,
cannot contain all elements of $E_i$, showing that $i\not \in
c(u)$. This proves that $c(u) \cap c(v) =\emptyset$. Therefore,
Theorem \ref{shomb} applies, finishing the proof.
}
\end{proof}
For a given graph $G$, if $u$ and $v$ are distinct vertices of $G$
and the neighborhood of $u$ is a subset of that of $v$, then the
graph $G$ is certainly not a vertex-critical graph. Note that in
the graph $SG(7,2,2)$, the neighborhood of the vertex
$(\{1,3\},\{4,5,6,7\})$ is a subset of that of the vertex
$(\{1,3\},\{2,4,5,6,7\})$. Hence, the graph $SG(m,n,k)$ in
general is not a vertex-critical graph.  This motivates us to
present the following definition.
\begin{defin}{Let $m, n,$ and $k$ be positive integers with $m\geq 2n$.
Define $SH(m,n,k)$ to be the induced subgraph of $H(m,n,k)$ whose
vertex set contains all  $k$-tuples $(A_1,\ldots ,A_k)\in
V(H(m,n,k))$ such that for any $1\leq r\leq k$,
$A_r={\displaystyle \cup_s} B_s$ and
$\overline{A_r}={\displaystyle \cup_t} C_t$, where $B_s$'s and
$C_t$'s are all $2$-stable $n$-subsets of $[m]$.
}
\end{defin}
One can check that $SH(m,n,k)$ has the property that for any two
distinct vertices $u,v \in V(SH(m,n,k)$, $N(u) \nsubseteq N(v)$
and $N(v) \nsubseteq N(u)$. Also, it is straightforward to see
that $SH(m,n,k)$ is the maximal subgraph of $SG(m,n,k)$ with the
aforementioned property. To prove this, we modify the graph
$SG(m,n,k)$ by performing the following {\bf WHILE}-loop.

\vspace{0.3cm}
{\bf WHILE} there exist two distinct vertices $u=(A_1,\ldots
,A_k)$ and $v=(B_1,\ldots ,B_k)$, where $N(u) \subseteq N(v)$,
then {\bf DO} the following: remove the vertex $u$.
\vspace{0.3cm}

We claim that in the {\bf WHILE-loop} algorithm when the input is
the graph $SG(m,n,k)$ with $m\geq 2n$, then the output is the
graph $SH(m,n,k)$. To show this, note that in the {\bf
WHILE}-loop each time we search in the new graph for the "bad"
vertex $u$. So a vertex $u$ may be good at the beginning, and
become bad later. Suppose that {\bf WHILE-loop} is not completed
yet. In the last graph obtained from the {\bf WHILE-loop}
algorithm, let $i$ be the greatest positive integer for which
there exists at least a vertex $u=(A_1,\ldots ,A_k)\in
V(SG(m,n,k))$ such that $\overline{A_i}$ is not a union of
$2$-stable $n$-subsets of $[m]$. Note that as $|A_1|=n$, it is
easy to verify that $\overline{A_1}$ is a union of $2$-stable
$n$-subsets of $[m]$, and hence $i \geq 2$. Also, by the
assumption, for any $i < j$, $\overline{A_j}$ is a union of
$2$-stable $n$-subsets of $[m]$. Set $v\isdef (A_1,\ldots
,A_{i-1},A_i\cup B, A_{i+1},\ldots ,A_k)$, where
$$B\isdef \{j| j\in \overline{A_i}\ {\rm and }\ j \ \ {\rm does\ not \ appear\ in\ any}\ 2{\rm -stable}\ n{\rm -subsets\
of}\ \overline{A_i}\}.$$ For any $j \in B$, since
$A_{i-1}\subseteq \overline{A_i}$ and that $A_{i-1}$ is a union
of $2$-stable $n$-subsets of $[m]$, it is easy to show that
$\{j-1,j+1\}\subseteq A_{i-1}\subseteq \overline{A_i}$ (mod m).
Therefore, $A_i\cup B$ is a union of $2$-stable $n$-subsets of
$[m]$. Also, by considering the assumption, we should have
$B\subseteq A_{i+2}$. Thus, $v\in V(SG(m,n,k))$ and also $N(u)
\subseteq N(v)$. Consequently, when the {\bf WHILE}-loop is
completed, we obtain the graph $SH(m,n,k)$. Also, this shows that
$SH(m,n,k)$ and $SG(m,n,k)$ are homomorphically equivalent. In
view of the above observation, we suggest the following question.

\begin{qu}\label{cri}
Let $m, n,$ and $k$ be positive integers with $m\geq 2n$. Is it
true that the graph $SH(m,n,k)$ is a vertex-critical graph?
\end{qu}

The problem whether the circular chromatic number and the
chromatic number of the Kneser graphs and the Schrijver graphs are
equal has received attention and has been studied in several
papers \cite{DAHA3, HAZH, jhs, LILI, ME, SITA}. Johnson, Holroyd,
and Stahl \cite{jhs} proved that $\chi_c({\rm KG}(m,n))=\chi({\rm
KG}(m,n))$ if $m\leq 2n+2$ or $n=2$. They also conjectured that
the equality holds for all Kneser graphs.

\begin{con}
\label{jhsconj} {\rm \cite{jhs}} For all $m \geq 2n+1$,
$\chi_c({\rm KG}(m,n))=\chi({\rm KG}(m,n))$.
\end{con}

It was shown in \cite{HAZH} that if $m \geq 2n^2(n-1)$, then the
circular chromatic number of ${\rm KG}(m,n)$ is equal to its
chromatic number. Later, it was proved independently in \cite{ME,
SITA} that  $\chi({\rm KG}(m,n))=\chi_c({\rm KG}(m,n))=m-2n+2$
whenever $m$ is an even natural number. Also in \cite{BAST, SITA},
it was shown that $\chi(H(m,1,k))=m$. Simonyi and Tardos
\cite{SITA} used the fact that $\hom[SG(a,b),H(m,1,k)]\not
=\emptyset$, where $a-2b+2=m$, and hence $m-1$ is a lower bound
for the co-index of the box complex of $H(m,1,k)$. For definition
of the box complex and more about this concept refer to
\cite{SITA}.

\begin{alphthm}{\rm (\cite{ME}, \cite{SITA})}\label{box}
If ${\rm coind}(B_0(G))$ is odd for a graph $G$, then
$\chi_c(G)\geq {\rm coind}(B_0(G))+1$.
\end{alphthm}

It was shown in \cite{SITA} that circular chromatic number and
chromatic number of $H(m,1,k)$ are equal.

\begin{thm}\label{cirhel}
Let $m, n,$ and $k$ be positive integers, where $m\geq 2n$ and $m$
is an even positive integer. Then, $\chi_c(SG(m,n,k))=
\chi_c(H(m,n,k))=m-2n+2$. Furthermore, $\chi_c(SH(m,n,k))=m-2n+2$.
\end{thm}
\begin{proof}
{As proved in Theorem \ref{chrom}, if $a-2b=m-2n$ and
$a=2(k-1)m(m-2n+2)+m$, then $\hom[SG(a,b),SG(m,n,k)]\not
=\emptyset$. This implies ${\rm coind}(B_o(SG(a,b))\leq {\rm
coind}(B_o(SH(m,n,k))$. Also, it is well known that ${\rm
coind}(B_o(SG(a,b))=a-2b+1$. Thus, by Theorem \ref{box}, we have
$\chi_c(SG(m,n,k))=\chi_c(H(m,n,k))=m-2n+2$. Also, two graphs
$SH(m,n,k)$ and $SG(m,n,k)$ are homomorphically equivalent. Thus,
$\chi_c(SH(m,n,k))=m-2n+2$.}
\end{proof}

In \cite{ME, SITA}, the authors made use of Theorem \ref{box} to
prove that $\chi_c((SG(a,b))=\chi((SG(a,b))$ provided that $a$ is
an even positive integer. In view of
$\chi_c((SG(a,b))=\chi((SG(a,b))$, where $a$ is an even integer
number, one can present an alternate proof of Theorem
\ref{cirhel}. However, note that the equality ${\rm
coind}(B_o(H(m,n,k))=m-2n+1$ provides more information about the
colorings of the helical graph $H(m,n,k)$ (see \cite{SITA,
SITA2}).

It was conjectured in \cite{LILI} and proved in \cite{HAZH}, that
for every fixed $n$, there is a threshold $t(n)$ such that
$\chi_c(SG(m,n))=\chi(SG(m,n))$ for all  $m \geq t(n)$. Note that
$H(3,1,2)$ is the nine cycle and that $\chi_c(H(3,1,2))={9 \over
4}$. Hence, the following question arises naturally.

\begin{qu}\label{PR0}
Given positive integers $n$ and $k$, does there exist a number
$t(n,k)$ such that the equality
$\chi_c(SH(m,n,k))=\chi_c(H(m,n,k))=\chi(H(m,n,k))=m-2n+2$ holds
for all $m \geq t(n,k)$?
\end{qu}
\section{Homomorphism to Odd Cycles}
In this section, we investigate the problem of existence of
homomorphisms to odd cycles. A graph $H$ is said to be a
subdivision of a graph $G$ if $H$ is obtained from  $G$ by
subdividing some of the edges. The graph $S_t(G)$ is said to be
the $t$-subdivision of a graph $G$ if $S_t(G)$ is obtained from
$G$ by replacing each edge by a path with exactly $t$ inner
vertices. Note that $S_0(G)$ is isomorphic to $G$. In the
following theorem, we prove that a homomorphism to $(2k+1)$-cycle
exists if and only if the chromatic number of (2k+1)st power of
$S_2(G)$ is less than or equal to 3.

\begin{thm}\label{ocy} Let $G$ be a graph with odd-girth at least
$2k+1$. Then, $\chi(S_2(G)^{(2k+1)}) \leq 3$ if and only if
$\hom[G,C_{2k+1}]\not =\emptyset$.
\end{thm}
\begin{proof}{
First, if there exists a homomorphism from $G$ to $C_{2k+1}$, then
it is obvious to see that there is a homomorphism from $S_2(G)$
to $C_{6k+3}=H(3,1,k+1)$. In view of Theorem \ref{homb}, we have
$\chi(S_2(G)^{(2k+1)})\leq 3$.

Next, if $\chi(S_2(G)^{(2k+1)}) \leq 3$, then $\hom[S_2(G),
C_{6k+3}]\not =\emptyset$. Consequently, $\hom[S_2(G)^{(3)},
C_{6k+3}^{(3)}]\not =\emptyset$. Also, it is easy to verify that
$G$ is a subgraph of $S_2(G)^{(3)}$ and that there is a
homomorphism from $C_{6k+3}^{(3)}$ to $C_{2k+1}$. Therefore, we
have $\hom[G,C_{2k+1}]\not =\emptyset$.}
\end{proof}

Considering Theorem \ref{ocy}, it is worth to study the following
question.

\begin{qu}\label{ChromThick}
Let $G$ be a non-bipartite graph. What is the value of
$$\sup\{{2k+1 \over 2t+1}| \chi(S_{2t}(G)^{(2k+1)})=\chi(G),
{2k+1 \over 2t+1}< og(G) \}?$$
\end{qu}

 In \cite{NES}, Ne\v set\v ril posed the Pentagon problem.
\begin{prb}\label{PR1} {\rm Ne\v set\v ril's
Pentagon Problem
\cite{NES}}\\
If $G$ is a cubic graph of sufficiently large girth, then
$\hom[G,C_{_{5}}] \not = \emptyset$.
\end{prb}

It should be noted that if in the problem $C_{5}$ is replaced by
$C_{3}$, then the problem holds; and in fact it is a quick
consequence of Brook's theorem. On the other hand, the problem is
known to be false if one replaces $C_{5}$ by $C_{11}$, $C_{9}$ or
$C_{7}$ \cite{HA,KNS,WO}.\\

In view of Theorem \ref{ocy}, it is possible to rephrase the
Pentagon Problem as follows.

\begin{qu}
Let $G$ be a cubic graph of sufficiently large girth, is it true
that $\chi(S_2(G)^{(5)}) \leq 3$?
\end{qu}

 If the Pentagon problem holds, then it follows from
 Lemma~\ref{M2} that there exists a number
$g_{_{0}}$ with the property that the chromatic number of the
third power of any cubic graph with girth larger than $g_{_{0}}$
is less than six.

\begin{qu}{\rm \cite{DAHA4}}
\label{PR2} Is it true that for any natural number $g_{_{0}}$,
there exists a cubic graph $G$ whose girth is larger than
$g_{_{0}}$ and $\chi(G^{^{(3)}}) \geq 6$?
\end{qu}

It is interesting to find $\displaystyle \max_{g(G)\geq g}
\chi(G^{(3)})$, where maximum is taken over all cubic graphs with
girth at least $g$. It should be noted that by Brook's theorem
this maximum is less than or equal to 16. In view of Theorem
\ref{homb}, the following question is equivalent to question
\ref{PR2}.

\begin{qu}
\label{PR3} Is it true that for any natural number $g_{_{0}}$,
there exists a cubic graph $G$ whose girth is larger than
$g_{_{0}}$ and $\hom[G,H(5,1,2)]=\emptyset$?
\end{qu}

Note that $H(3,1,2)$ is the nine cycle. It was proved in \cite{WO}
that the above question has an affirmative answer when $H(5,1,2)$
is replaced by $H(3,1,2)$. This motivates us to suggest the
following question.

\begin{qu}
\label{PR4} Is it true that for any natural number $g_{_{0}}$,
there exists a cubic graph $G$ whose girth is larger than
$g_{_{0}}$ and $\hom[G,H(4,1,2)]=\emptyset$?
\end{qu}

The fractional chromatic number of graphs with odd-girth greater
than $3$ has been studied in several papers \cite{HATZH, HETH}.
Heckman and Thomas \cite{HETH} posed the following conjecture.

\begin{con}{\rm \cite{HETH}}\label{frt}
Every triangle free graph with maximum degree at most $3$ has the
fractional chromatic number at most ${14\over 5}$.
\end{con}

The helical graphs bound high girth graphs. Thus, it may be
interesting to compute their fractional chromatic number and
their local chromatic number.

\begin{qu}
Let $m, n,$ and $k$ be positive integers with $m\geq 2n$. What are
the values of $\chi_f(H(m,n,k))$ and $\psi(H(m,n,k))?$
\end{qu}

Let ${\cal P}_{2k+1}$ be the class of planar graphs of odd-girth
at least $2k +1$. Naserasr \cite{RN} posed an upper bound for the
chromatic number of planar graph powers as follows.

\begin{con}{\rm \cite{RN}}
For every $G\in {\cal P}_{2k+1}$ we have $\chi(G^{(2k-1)})\leq
2^{2k}$.
\end{con}

Again in view of  Theorem \ref{homb}, one can rephrase Naserasr's
conjecture in terms of the helical graphs. The following
conjecture is Jaeger's modular orientation conjecture restricted
to planar graphs.

\begin{con}{\rm Jaeger's Conjecture \cite{JA}}\\
Every planar graph with girth  at least $4k$ has a homomorphism
to $C_{2k+1}$.
\end{con}

Considering Theorem \ref{ocy}, one can reformulate Jaeger's
conjecture as follows.

\begin{con}
Let $P$ be a planar graph with girth  at least $4k$. Then, we have
$\chi(S_2(P)^{(2k+1)})\leq 3$.
\end{con}
\ \\
{\bf Acknowledgement:} This paper was written during the
sabbatical leave of the author in Zurich University. He wishes to
thank J. Rosenthal for his hospitality. Also, the author wishes
to thank an anonymous referee, G. Simonyi, C. Tardif, G. Tardos,
B.R. Yahaghi, and X. Zhu who drew the author's attention to the
references \cite{BAST}, \cite{GYJE} and \cite{HETH} and for their
useful comments.

\end{document}